\documentclass[12pt]{amsart}
\usepackage{amsmath,amssymb,amsbsy,amsfonts,latexsym,amsopn,amstext,
                                               amsxtra,euscript,amscd}

\newfont{\teneufm}{eufm10}
\newfont{\seveneufm}{eufm7}
\newfont{\fiveeufm}{eufm5}
\newfam\eufmfam
                \textfont\eufmfam=\teneufm \scriptfont\eufmfam=\seveneufm
                \scriptscriptfont\eufmfam=\fiveeufm
\def\bbbc{{\mathchoice {\setbox0=\hbox{$\displaystyle\rm C$}\hbox{\hbox
to0pt{\kern0.4\wd0\vrule height0.9\ht0\hss}\box0}}
{\setbox0=\hbox{$\textstyle\rm C$}\hbox{\hbox
to0pt{\kern0.4\wd0\vrule height0.9\ht0\hss}\box0}}
{\setbox0=\hbox{$\scriptstyle\rm C$}\hbox{\hbox
to0pt{\kern0.4\wd0\vrule height0.9\ht0\hss}\box0}}
{\setbox0=\hbox{$\scriptscriptstyle\rm C$}\hbox{\hbox
to0pt{\kern0.4\wd0\vrule height0.9\ht0\hss}\box0}}}}
\def\bbbq{{\mathchoice {\setbox0=\hbox{$\displaystyle\rm
Q$}\hbox{\raise 0.15\ht0\hbox to0pt{\kern0.4\wd0\vrule
height0.8\ht0\hss}\box0}} {\setbox0=\hbox{$\textstyle\rm
Q$}\hbox{\raise 0.15\ht0\hbox to0pt{\kern0.4\wd0\vrule
height0.8\ht0\hss}\box0}} {\setbox0=\hbox{$\scriptstyle\rm
Q$}\hbox{\raise 0.15\ht0\hbox to0pt{\kern0.4\wd0\vrule
height0.7\ht0\hss}\box0}} {\setbox0=\hbox{$\scriptscriptstyle\rm
Q$}\hbox{\raise 0.15\ht0\hbox to0pt{\kern0.4\wd0\vrule
height0.7\ht0\hss}\box0}}}}
\def\bbbt{{\mathchoice {\setbox0=\hbox{$\displaystyle\rm
T$}\hbox{\hbox to0pt{\kern0.3\wd0\vrule height0.9\ht0\hss}\box0}}
{\setbox0=\hbox{$\textstyle\rm T$}\hbox{\hbox
to0pt{\kern0.3\wd0\vrule height0.9\ht0\hss}\box0}}
{\setbox0=\hbox{$\scriptstyle\rm T$}\hbox{\hbox
to0pt{\kern0.3\wd0\vrule height0.9\ht0\hss}\box0}}
{\setbox0=\hbox{$\scriptscriptstyle\rm T$}\hbox{\hbox
to0pt{\kern0.3\wd0\vrule height0.9\ht0\hss}\box0}}}}
\def\bbbs{{\mathchoice
{\setbox0=\hbox{$\displaystyle     \rm S$}\hbox{\raise0.5\ht0\hbox
to0pt{\kern0.35\wd0\vrule height0.45\ht0\hss}\hbox
to0pt{\kern0.55\wd0\vrule height0.5\ht0\hss}\box0}}
{\setbox0=\hbox{$\textstyle        \rm S$}\hbox{\raise0.5\ht0\hbox
to0pt{\kern0.35\wd0\vrule height0.45\ht0\hss}\hbox
to0pt{\kern0.55\wd0\vrule height0.5\ht0\hss}\box0}}
{\setbox0=\hbox{$\scriptstyle      \rm S$}\hbox{\raise0.5\ht0\hbox
to0pt{\kern0.35\wd0\vrule height0.45\ht0\hss}\raise0.05\ht0\hbox
to0pt{\kern0.5\wd0\vrule height0.45\ht0\hss}\box0}}
{\setbox0=\hbox{$\scriptscriptstyle\rm S$}\hbox{\raise0.5\ht0\hbox
to0pt{\kern0.4\wd0\vrule height0.45\ht0\hss}\raise0.05\ht0\hbox
to0pt{\kern0.55\wd0\vrule height0.45\ht0\hss}\box0}}}}
\def\bbbz{{\mathchoice {\hbox{$\sf\textstyle Z\kern-0.4em Z$}}
{\hbox{$\sf\textstyle Z\kern-0.4em Z$}} {\hbox{$\sf\scriptstyle
Z\kern-0.3em Z$}} {\hbox{$\sf\scriptscriptstyle Z\kern-0.2em Z$}}}}

\newtheorem{theorem}{Theorem}
\newtheorem{lemma}[theorem]{Lemma}

\newtheorem{cor}[theorem]{Corollary}

\newtheorem{rem}[theorem]{Remark}

\def\squareforqed{\hbox{\rlap{$\sqcap$}$\sqcup$}}
\def\qed{\ifmmode\squareforqed\else{\unskip\nobreak\hfil
\penalty50\hskip1em\null\nobreak\hfil\squareforqed
\parfillskip=0pt\finalhyphendemerits=0\endgraf}\fi}

\def\cA{{\mathcal A}}
\def\cB{{\mathcal B}}

\def\cE{{\mathcal E}}

\def\cP{{\mathcal P}}

\def\cU{{\mathcal U}}

\def\cX{{\mathcal X}}
\def\cY{{\mathcal Y}}

\def \sf {\mathfrak s}

\def\ZK{\Z_\K}

\def\Res{{\mathrm{Res}}}

\newcommand{\ignore}[1]{}

\def\vec#1{\mathbf{#1}}

\def \C{\mathbb{C}}
\def \F{\mathbb{F}}
\def \K{\mathbb{K}}

\def \Z{\mathbb{Z}}

\def \R{\mathbb{R}}
\def \Q{\mathbb{Q}}

\def \Z{\mathbb{Z}}

\def\mand{\qquad\mbox{and}\qquad}

\def\\{\cr}
\def\({\left(}
\def\){\right)}
\def\fl#1{\left\lfloor#1\right\rfloor}

\def\eps{\varepsilon}

\begin{document}

\title[Congruences with Products  from Short Intervals]{On Congruences with Products of Variables from Short Intervals and Applications}

\author[J.~Bourgain]{Jean~Bourgain}
\address{Institute for Advanced Study,
Princeton, NJ 08540, USA} \email{bourgain@ias.edu}

\author[M.~Z.~Garaev]
{Moubariz~Z.~Garaev}
\address{Centro de Ciencias Matem\'{a}ticas, Universidad Nacional Aut\'onoma de M\'{e}xico,
C.P. 58089, Morelia, Michoac\'{a}n, M\'{e}xico}
\email{garaev@matmor.unam.mx}

\author[S.~V.~Konyagin]{Sergei V.~Konyagin}
\address{Steklov Mathematical Institute,
8, Gubkin Street, Moscow, 119991, Russia} \email{konyagin@mi.ras.ru}

\author[I.~E.~Shparlinski]{Igor E.~Shparlinski}
\address{Department of Computing, Macquarie University,
Sydney, NSW 2109, Australia} \email{igor.shparlinski@mq.edu.au}

\begin{abstract}
We obtain upper bounds on the number of solutions to congruences of
the type
$$
(x_1+s)\ldots(x_{\nu}+s)\equiv  (y_1+s)\ldots(y_{\nu}+s)\not\equiv0
\pmod p
$$
modulo a prime $p$ with variables from some short intervals.
We give some applications of our results and in particular improve
several recent estimates of J.~Cilleruelo and
M.~Z.~Garaev on exponential congruences and on cardinalities of products
of short intervals,  some double character sum estimates of J. Friedlander
and H. Iwaniec and some results of M.-C.~Chang
and A.~A.~Karatsuba on character sums twisted with the divisor function.
\end{abstract}

%

\maketitle

\section{Introduction}

For a prime $p$, let $\F_p$ be the field of residues modulo $p$.
Also, denote $\F_p^*=\F_p\setminus\{0\}$. For integers $h$ and $\nu
\ge 1$ and elements $s \in \F_p$ and $\lambda  \in \F_p^*$, we
denote by $J_{\nu}(p,h,s;\lambda)$ the number of solutions of the
congruence
\begin{equation}
\label{eq:asym cong x,y}
\begin{split}
(x_1+s)\ldots(x_{\nu}+s)\equiv & \lambda \pmod p,\\
1\le x_1,\ldots,&x_{\nu}\le h.
\end{split}
\end{equation}

For large values of $h$, one can use bounds of Kloosterman sums (for
$\nu =2,3$) and multiplicative character sums (for $\nu \ge 4$) to
obtain various asymptotic formulas for  $J_{\nu}(p,h,s;\lambda)$,
see~\cite{Gar,GarGar,LeBd,Shp1,Shp2}. However, this approach does not give
any nontrivial estimates for small values of $h$, and thus Chan and
Shparlinski~\cite{ChanShp}, for $\nu=2$, have employed methods of
additive combinatorics, namely some results of Bourgain~\cite{Bour},
in order to obtain a nontrivial upper bound on
$J_{\nu}(p,h,s;\lambda)$ for any $h$.

Cilleruelo and Garaev~\cite{CillGar} have substantially improved
 the bounds of~\cite{ChanShp}, obtained several results for $\nu =3$
 and also suggested
several conjectures.

Recently, motivated by some applications to certain algorithmic
problems, new results on $J_{\nu}(p,h,s;\lambda)$ have been given by
Bourgain, Garaev, Konyagin and Shparlinski~\cite{BGKS}. In
particular, it is shown in~\cite{BGKS} that for
$$
h<p^{1/(\nu^2-1)},
$$
we have the bound
\begin{equation}
\label{eq:CillCar Conj} J_{\nu}(p,h,s;\lambda)<
\exp\(c(\nu)\frac{\log h}{\log\log h}\),
\end{equation}
uniformly over $s \in \F_p$  and  $\lambda  \in \F_p^*$,
 where  $c(\nu)$ depends only on $\nu=2,3, \ldots$.
In particular, for $\nu = 4$ the bound~\eqref{eq:CillCar Conj}
answers the  open question from~\cite[Section~6]{CillGar}.

Here we use and develop further some ideas of~\cite{BGKS} and study
a symmetric version of the congruence~\eqref{eq:asym cong x,y}. More
precisely, for a prime $p$, integers $h$ and $\nu \ge 1$ and an
element   $s \in \F_p$, we study  the number of solutions
$K_{\nu}(p,h,s)$ of the congruence
\begin{equation}
\label{eq:cong x,y}
\begin{split}
(x_1+s)\ldots(x_{\nu}+s)\equiv & (y_1+s)\ldots(y_{\nu}+s)\not\equiv0 \pmod p,\\
1\le x_1,\ldots,&x_{\nu}, y_1,\ldots,y_{\nu}\le h.
\end{split}
\end{equation}

We note that for $\nu=2$ this question, and its generalizations to
residue rings and arbitrary finite fields,  has been considered in
a number of works~\cite{ACZ,CochSih,FrIw,Kon}. So, although our
argument works for $\nu = 2$ as well, here we concentrate on the
case $\nu\ge 3$.

We believe that our results are of independent interest and  then
may also be used to improve some previous results. For example,
Corollary~\ref{cor:ProdSet Fp} extends the range of $h$ under which
a similar result is obtained in~\cite{BGKS}.

Furthermore, it is easy to see that bounds on  $K_{\nu}(p,h,s)$ can be reformulated
as statements about moments of character sums over the intervals
$[s,s+h]$, for example, see Lemma~\ref{lem:Kss} below. As such,
they also
complement various other results of the type which can be found in
the literature, see~\cite{ACZ,CochSih,CoZh,FrIw} and references therein.
Using the ideas behind our estimates of $K_{\nu}(p,h,s)$ we estimate
the number of solutions of  several other congruences of similar
form which in turn leads to improvements of the  bounds
\begin{itemize}
\item of Cilleruelo and Garaev~\cite[Corollary~3]{CillGar} on the
number of solutions to exponential congruences in small intervals;
\item of Friedlander and
Iwaniec~\cite{FrIw2} on double character sums over subsets of intervals;
\item of Chang~\cite{Chang1} and Karatsuba~\cite{Kar1, Kar4} on the
character sums with the divisor function.
\end{itemize}

\section{Resultant Bound}

For positive integers  $m,n$ with $m,n\ge2$ and
$\sigma\in\R$, we define the $(m+n-2)\times (n-1)$
circulant matrix   $A(m, n,\sigma)$  as follows:
\begin{equation*}
\(
  \begin{array}{cccccccc}
    \sigma & \sigma+1 & \ldots & \sigma + m-1 & 0 & 0 & \ldots & 0 \\
    0 & \sigma  & \ldots  & \sigma +m-2 &\sigma + m-1 & 0 & \ldots & 0  \\
    \ldots & \ldots & \ldots  & \ldots & \ldots & \ldots & \ldots & \ldots \\
    0 & \ldots & 0 & \sigma &\sigma+1 & \ldots  & \ldots &\sigma + m-1\\
  \end{array}
\).
\end{equation*}
We mark all elements located in the intersection of $i$-th row
and $j$-th column if $i\le j\le i+m-1$.
Note that all unmarked elements are zeros and,
conversely, for $\sigma>0$ all zeros are unmarked.

\begin{lemma}
\label{lem:DeterMagic} Let $m,n\ge 2$ be integers and
$\sigma,\vartheta\in\R$. If in the  $(m+n-2)\times (m+n-2)$ matrix
$$
X(m, n)=\(
  \begin{array}{c}
    A(m,n,\sigma)\\
    A(n,m,\vartheta) \\
  \end{array}
\)
$$
we select $m+n-2$ marked elements such that  each row and each column
contains exactly one selected element then the sum of the selected
elements is always equal to
$$
 \Sigma(m,n,\sigma,\vartheta)
=(m-1+\sigma)(n-1+\vartheta) - \sigma\vartheta.
$$
\end{lemma}

\begin{proof}
Let
$$
X(m, n)=(x_{i,j})_{1\le i,j\le m+n-2},
$$
where $i$ indicates the row. Since the sum of the diagonal elements
of $X(m, n)$ is equal to $(m-1+\sigma)(n-1+\vartheta) - \sigma\vartheta$,
it suffices to prove that the sum of the selected elements does not
depend on the choice of selection. To see this, we transform the
matrix $X(m, n)$ into  a matrix
$$
Y(m, n)=(y_{i,j})_{1\le i,j\le m+n-2}
$$
as follows
\begin{itemize}
\item  If $x_{i,j}$ is unmarked, then we put $y_{i,j} =0$
\item If $x_{i,j}$ is marked, then we put
$$y_{i,j} = \left\{
\begin{array}{ll}
x_{i,j} + 2i-\sigma, & \text{ for } 1\le i\le n-1,\\
x_{i,j} + 2i-n+1-\vartheta, & \text{ for }   n\le i\le m+n-2.
\end{array}
\right.
$$
\end{itemize}

Since the selected elements occur in each row exactly once, from this
transformation of $X(m,n)$ into $Y(m,n)$ the sum of the
elements at the marked positions changes only by
$$\sigma_1 =\sum_{i =1}^{n-1} ( 2i-\sigma)  +\sum_{i=n}^{m+n-2}
(2i-n +1-\vartheta)
$$
and in particular does not depend on the choice of the selection.
Therefore, it suffices to show that the sum of corresponding selected
elements of $Y(m,n)$ does not depend on the choice of selection.
But this follows from the observation that when $x_{ij}$ is marked, we
have that
$$
y_{i,j}=i+j.
$$
Hence, the sum of the corresponding selected elements of $Y(m,n)$
is equal to
$$
\sigma_2 = 2(1+\ldots +(m+n-2)) = (m+n-1)(m+n-2)
$$
and does not depend on the choice of selection. Since $\sigma_2 -
\sigma_1 = \sigma$, the result now follows.
\end{proof}

We  need the following simple statement.

\begin{lemma}
\label{lem:neq:sigma} Let $M\ge m\ge 2$, $N\ge n\ge 2$ be integers,
$\sigma+M-m\ge0$, $\vartheta+N-n\ge0$. Assume also that one of the
following conditions hold:
\begin{itemize}
\item[(i)] $\sigma\ge0$;
\item[(ii)]$\vartheta\ge0$;
\item[(iii)] $\sigma+\vartheta\ge -1$.
\end{itemize}
Then $\Sigma(M,N,\sigma,\vartheta)\ge \Sigma(m,n,\sigma+M-m,\vartheta+N-n)$.
\end{lemma}

\begin{proof} Clearly,
\begin{equation*}
\begin{split}
\Sigma(M,N,\sigma,\vartheta) &- \Sigma(m,n,\sigma+M-m,\vartheta+N-n)\\
=(\sigma+M-m)&(\vartheta+N-n) - \sigma\vartheta \ge0.
\end{split}
\end{equation*}
Since either of the conditions~(i)--(iii) implies
$$
(\sigma+M-m)(\vartheta+N-n)\ge\sigma\vartheta,
$$
 the result follows.
\end{proof}

\begin{cor}
\label{cor:DeterMagic} Let $H\ge1$,
$\sigma,\vartheta\in\R$, and let  $M,N\ge2$ be fixed integers.
Assume that either of the conditions (i)--(iii) of Lemma~\ref{lem:neq:sigma}
is satisfied. Let
$P_1(Z)$ and $P_2(Z)$ be non-constant polynomials,
$$
P_1(Z)=\sum_{i=0}^{M-1}a_{i}Z^{M-1-i} \mand P_2(Z)=
\sum_{i=0}^{N-1}b_{i}Z^{N-1-i}
$$
such that
$$ |a_{i}|<H^{i+\sigma}, \quad i =0, \ldots, M-1,$$
$$ |b_{i}|<H^{i+\vartheta}, \quad i =0, \ldots, N-1.$$
Then
$$
\Res(P_1, P_2)\ll H^{\Sigma(M,N,\sigma,\vartheta)},
$$
where the implicit constant in $\ll$ depends only on $M$ and $N$.
\end{cor}

\begin{proof}
Let $m-1=\deg P_1$ and  $n-1=\deg P_2$. We have $2\le m\le M$,
$2\le n\le N$. The inequalities $|a_{M-m}|\ge1$ and $|b_{N-n}|\ge1$
imply $\sigma+M-m\ge0$ and $\vartheta+N-n\ge0$, respectively. We recall that
$$
\Res(P_1, P_2)=\det\(
\begin{array}{c}  A\\  B\\ \end{array}\),
$$
where
\begin{equation*}
A=\(
  \begin{array}{cccccccc}    a_{M-m} &  \ldots &  a_{M-2} &a_{M-1} & 0 & 0 & \ldots & 0 \\
    0 & a_{M-m}&  \ldots &  a_{M-2} &a_{M-1} & 0 & \ldots & 0  \\
    \ldots & \ldots & \ldots  & \ldots & \ldots & \ldots & \ldots & \ldots \\
    0 & \ldots & 0 & a_{M-m} &  \ldots & \ldots &  a_{M-2} &a_{M-1}\\
  \end{array}
\)
\end{equation*}
and
\begin{equation*}
B=\(
  \begin{array}{cccccccc}   b_{N-n} & \ldots & b_{N-2}  &  b_{N-1}  & 0 & 0 & \ldots & 0 \\
    0 & b_{N-n} & \ldots & b_{N-2}  &  b_{N-1} & 0 & \ldots & 0  \\
    \ldots & \ldots & \ldots  & \ldots & \ldots & \ldots & \ldots & \ldots \\
    0 & \ldots & 0 & b_{N-n} & \ldots & \ldots & b_{N-2}  &  b_{N-1}\\
  \end{array}
\)
\end{equation*}
are $(m+n-2)\times (n-1)$ and $(m+n-2)\times (m-1)$ matrices,
respectively. The result now follows from the representation of the
determinant by sums of products of its elements and
Lemmas~\ref{lem:DeterMagic} and~\ref{lem:neq:sigma}.
\end{proof}

\section{More General Congruences}

To estimate $K_{\nu}(p,h,s)$ we sometimes have to study a more
general congruence. For a prime $p$, integers $h$ and $\nu \ge 1$
and  a vector $\vec{s} = (s_1, \ldots, s_\nu) \in \F_p$ we denote by
$K_{\nu}(p,h,\vec{s})$ the number of solutions of the congruence
\begin{equation*}
\begin{split}
(x_1+s_1)\ldots(x_{\nu}+s_{\nu})\equiv & (y_1+s_1)\ldots(y_{\nu}+s_{\nu})\not\equiv0 \pmod p,\\
1\le x_1,\ldots,&x_{\nu}, y_1,\ldots,y_{\nu}\le h.
\end{split}
\end{equation*}

This following simple statement relates $K_{\nu}(p,h,\vec{s})$ and
$K_{\nu}(p,h,s_j)$, $j =1, \ldots, \nu$.

\begin{lemma}
\label{lem:Kss} We have
$$
K_{\nu}(p,h,\vec{s}) \le \prod_{j=1}^\nu  K_{\nu}(p,h,s_j)^{1/\nu}
$$
\end{lemma}

\begin{proof}
Using the orthogonality of multiplicative characters, we write
\begin{equation*}
\begin{split}
K_{\nu}(p,h,\vec{s}) & = \frac{1}{p-1} \sum_{1\le
x_1,\ldots,x_{\nu}, y_1,\ldots,y_{\nu}\le h}\hskip -30pt{}^* \hskip
30pt
\sum_{\chi} \chi\(\frac{(x_1+s_1)\ldots(x_{\nu}+s_{\nu})}{(y_1+s_1)\ldots(y_{\nu}+s_{\nu})}\)\\
& = \frac{1}{p-1} \sum_{\chi}  \prod_{j=1}^\nu \sum_{1\le x_j,y_j\le
h}\hskip -10pt{}^* \hskip 10pt \chi\(\frac{x_j+s_j}{y_j+s_j}\),
\end{split}
\end{equation*}
where $\chi$ runs through all multiplicative characters modulo $p$
and $\Sigma^*$ indicates that summation does not involve $y_j\equiv
-s_j \pmod p$. Using the H{\"o}lder inequality, we obtain the
desired inequality.
\end{proof}

\section{Linear Congruences with Many Solutions}

We  need the following result, which in turn improves one
of the results from~\cite{ChangCillGHShZ}.

\begin{lemma}
\label{lem:LinearCongr} Let   $\gamma \in  (0,1)$ and let $I$ and
$J$ be two intervals containing $h$ and $H$ consecutive integers,
respectively, and  such that
$$
h\le H<\frac{\gamma \,p}{15}.
$$
Assume that for some integer $s$ the congruence
$$
y\equiv sx\pmod p
$$
has at least $\gamma h+1$ solutions in $x\in I$, $y\in J$.
Then there exist integers $a$ and $b$ with
$$
|a| \le \frac{H}{\gamma\,h},
\qquad 0<b\le \frac{1}{\gamma},
$$
such that
$$
s\equiv a/b \pmod p.
$$
\end{lemma}

\begin{proof}
We can assume that $s\not\equiv 0\pmod p$, as otherwise the statement
is trivial. Making a shift of the set $I\times J$ by the solution
$(x_0,y_0)$ of our congruence with the least $x_0$ (here we use a natural
ordering on $I$), without loss of generality we can assume that
$I\subseteq [0,h]$, $J\subseteq [-H,H]$. Since $s\not\equiv 0\pmod p$
and our congruence has a solution with $x\not=0$, there exist integers $a,b$ such that
$$
s\equiv a/b \pmod p,\quad 0<|a|\le H, \quad 0<b\le h, \quad \gcd(a,b)=1.
$$
Thus, the equation
$$
ax=by+pz
$$
has at least $\gamma h+1$ solutions in integer variables $x,y,z$ with $x\in I$, $y\in J$. We have
$$
|z|\le L,
$$
where
$$
L = \frac{|a|h+bH}{p}.
$$
We consider two cases, $L<1$ and $L\ge 1$.

{\it Case~1\/}: $L<1$. Then $z=0$ and we get that the equation $ax=by$ has at least $\gamma h+1$ solutions in $x\in I, y\in J$. Since $\gcd(a,b)=1$, we get that $x=bw$, $y=aw$ for some integer $w$ and this should hold for at least $\gamma h+1$ integers $w
$  (as there are at least $\gamma h+1$ pairs $(x,y)$). Therefore, $b\gamma h\le h$ and $|a|\gamma h\le H$ and the result follows.

{\it Case~2\/}: $L\ge 1$. Note that
$$
L\le \frac{2hH}{p}\le \frac{2\gamma h}{15}.
$$
Thus, by the pigeon-hole principle, there exists $z=z_0$ such that the equation
$ax=by+pz_0$
has at least $\gamma h/(3L)$ solutions in variables $x\in I,y\in J$. We fix one such solution $(x_0,y_0)\in I\times J$ and get that the equality
$$
a(x-x_0)=b(y-y_0),
$$
holds for at least $\gamma h/(3L)$ pairs $x,y$ with $|x-x_0|\le h, |y-y_0|\le H$. Since $\gcd(a,b)=1$, the equality implies
$$
x-x_0=bw, \quad y-y_0=aw,
$$
and this holds for at least $\gamma h/(3L)$ integers $w$. In particular, $|aw|\le H$ and $|bw|\le h$ for at least $\gamma h/(3L)$ integers $w$.  Clearly, one of these integers $w$ satisfies $|w|> \gamma h/(7L)$ and we therefore get
$$
|a|\gamma h<7LH,\quad b\gamma < 7L.
$$
Together with the definition of $L$, this implies that
$$
\gamma p  =\frac{\gamma |a|h+\gamma bH}{L}<14H,
$$
contradicting the condition of our lemma.
\end{proof}

\section{Congruences with Solution in Arbitrary Sets}

We  now use Lemma~\ref{lem:LinearCongr}
to obtain a version of Theorem~\ref{thm:GProdSet} below with $\nu = 2$,
which applies to exponential congruences with variables from short intervals.

\begin{lemma}
\label{lem:GProdSet nu=2} Let $\cX\subseteq [1,h]$ be a set of integers with
$h^3/(\# \cX)<0.002p$.
Then for the number of solutions $L(p,\cX;s) $ of the congruence
\begin{equation}
\label{eq:two}
(x_1+s)(x_2+s)\equiv (y_1+s)(y_2+s)\not\equiv 0  \pmod p,\quad
x_1,x_2,y_1,y_2\in \cX
\end{equation}
we have
$$
L(p,\cX;s) \le (\# \cX)^2 \exp\(C\log h/\log\log h\),
$$
where  $C$ is an absolute constant.
\end{lemma}

\begin{proof}
Clearly, it is enough to estimate the contribution $N$
to $L(p,\cX;s)$ of solutions of~\eqref{eq:two} with $x_i\not=y_j$,
$1 \le i,j \le 2$.

Let $X = \# \cX$. We also assume that
$$
N > X^2\exp\(c_0\frac{\log h}{\log\log h}\)
$$
for some large constant $c_0$ that is to be specified later. Observe that
the last inequality implies
$$
X > \exp\(c_0\frac{\log h}{\log\log h}\)
$$
due to the trivial estimate $N\le X^3$.

Note that for any $Z$ we have
$$
(x_1+Z)(x_2+Z)-(y_1+Z)(y_2+Z)=uZ-v,
$$
where
$$u=x_1+x_2-y_1-y_2,\quad v=y_1y_2-x_1x_2.$$

By the pigeon-hole principle we have
at least $N/X$ solutions of~\eqref{eq:two} with the same $x_1=x_1^*$. We claim
that any pair $(u,v)$ induced by these solutions occurs at most
$\exp\(c_0 \log h/\log\log h\)$ times for some constant
$c_0$. Indeed, fix a pair $(u,v)$ and take $Z=-x_1^*$. We get
\begin{equation}
\label{eq: repr M1} uZ-v = -(y_1-x_1^*)(y_2-x_1^*).
\end{equation}
The number of solutions to~\eqref{eq: repr M1} is bounded by
$\exp\(c_0\log h/\log\log h\)$. Each solution determines the numbers
$y_1,y_2$ and the polynomial $P$, and for each $y_1,y_2$ we retrieve $x_2$.
This proves the claim.

Therefore, there are at least $N\exp\(-c_0 \log h/\log\log h\)/X\ge X$ pairs
$(u,v)$ with
$$
0<|u|<2h,\quad 0<|v|<h^2,
$$
such that
$$
us\equiv v  \pmod p.
$$
We apply Lemma~\ref{lem:LinearCongr} (with $I=[-2h,2h]$, $J=[-2h^2,2h^2]$, $\gamma=X/(6h)$) and conclude that there are integers $a$ and $b$ satisfying conditions
\begin{equation}
\label{eq:cond_a,b} |a|\le 6h^2/X, \quad 0<|b|\le 6h/X, \quad s\equiv
a/b\pmod p.
\end{equation}
Now we multiply our original congruence
$$
(x_1+x_2-y_1-y_2)s+(x_1x_2-y_1y_2)\equiv 0\pmod p
$$
by $b$ and for $S=(x_1+x_2-y_1-y_2)a+(x_1x_2-y_1y_2)b$
we see that
$S\equiv 0\pmod p$.
Since $h^3/X<0.002 p$,  using~\eqref{eq:cond_a,b} we derive that $|S|<p$. Thus, $S=0$ and the congruence
is converted to an equality, giving
$$
(bx_1+a)(bx_2+a)=(by_1+a)(by_2+a),
$$
and the result follows from the bound on the divisor function.
\end{proof}

\begin{rem}  It is not difficult to show that the condition
$h^3/(\# \cX)<0.002p$ of Lemma~\ref{lem:GProdSet nu=2} can be
relaxed to $h^3/(\# \cX)<C_0p$, with any  constant $C_0>0$.
\end{rem}

The following result
is an extension of the well-known multiplicative energy estimate for pairs
of intervals, frequently needed in
character sum estimates, see~\cite[Theorem 3]{FrIw2}.
We use this result in the proof of Theorem~\ref{thm:9/20} below.

\begin{lemma}
\label{lem:prop1}
Let $A$ and $B$ be positive
integers with $AB\ll p$. Assume that $I$ is an interval consisting on $A$ consecutive integers, $\cY$ is a subset of an interval consisting on $B$ consecutive integers with $0\not\in \cY$. Then the number of solutions of the congruence
$$
x_1y_1\equiv x_2y_2\pmod p,\quad (x_1,x_2,y_1,y_2)\in I\times I\times \cY\times \cY
$$
is at most $(\# \cY)^2+A\#\cY p^{o(1)}$.
\end{lemma}

\begin{proof}
We can assume that $A<0.1p$ and $B<0.1p$, as otherwise the result becomes trivial.

Assume $I=\{\xi+1,\xi+2,\ldots, \xi+A\}$, $\cY\subseteq \{s+1,s+2,\ldots, s+B\}$. Let $\cY_0=\cY-\{s\}\subseteq [1,B]$. We have to estimate the number of solutions of
\begin{equation}
\label{eqn:Prop2Eq}
(\xi+x_1)(s+y_1)\equiv (\xi+x_2)(s+y_2)\pmod p
\end{equation}
with $1\le x_1,x_2\le A$,  $y_1,y_2\in \cY_0$.
For a given pair $y_1,y_2$, the number of solutions of~\eqref{eqn:Prop2Eq} with $1\le x_1,x_2\le A$ is clearly bounded by the number of solutions of
the congruence
$$
x_1(s+y_1)\equiv x_2(s+y_2)\pmod p
$$
with $|x_1|,|x_2|\le A$.
Thus, the number of solutions of the congruence~\eqref{eqn:Prop2Eq} is bounded
by the number of solutions of the congruence
\begin{equation}
\label{eqn:Prop2Reduced}
x_1(s+y_1)\equiv x_2(s+y_2)\pmod p,\quad
1\le |x_1|,|x_2|<A,\ y_1,y_2\in \cY_0,
\end{equation}
augmented by $(\#\cY_0)^2=(\#\cY)^2$.

Let $N$ be the number of solutions of~\eqref{eqn:Prop2Reduced}.
We  assume that $N\ge2\# \cY$
since otherwise there is nothing to prove.
For an appropriately fixed $y_1\in \cY_0$ we obtain at least $N/\# \cY -1\ge N/(2\# \cY)$
solutions with $y_2\not=y_1$ (recall that $s+y_1\not \equiv 0\pmod p$
so $y_1 = y_2$ implies $x_1 = x_2$). If for each pair $(u,v)$ of the form
\begin{equation}
\label{eq:pair uv}
(u,v) = (x_1 - x_2, x_2y_2 - x_1y_1)
\end{equation}
we specify the polynomial
$$
R_{u,v}(Z)=uZ - v  = x_1(Z+y_1)-x_2(Z+y_2),
$$
then we have
$$
R_{u,v}(-y_1)\equiv x_2(y_1-y_2)\pmod p.
$$
Since $|x_2(y_1-y_2)| \le 2AB \ll p$,
we get at most $p^{o(1)}$ possibilities for $x_2$ and $y_2$ and  hence for
$(x_1,x_2,y_2)$  (recall that $y_1$ is fixed and $y_1\not\equiv y_2\pmod p$).
 Thus, when $(x_1,x_2,y_2)$ runs through the set of solutions, we get at least  $Np^{o(1)}/\# \cY$ distinct polynomials $R_{u,v}(Z)$. Note
that
$$R_{u,v}(s) = us-v\equiv 0\pmod p
$$
for each pair $(u,v)$ of the form~\eqref{eq:pair uv}. Therefore,
there are at least $Np^{o(1)}/\# \cY$
solutions $(u,v)\in \Z\times \Z$ of the congruence $us-v\equiv 0\pmod p$
with $|u|\le 2A$, $|v|\le 2AB$. On the other hand, for any $u$ there
are $O(1)$ values of $v$ satisfying $us-v\equiv 0\pmod p$ since $AB\ll p$.
Thus, $Np^{o(1)}/\# \cY\ll A$, and the desired result follows.
\end{proof}

The following result is used in estimating character sums with the divisor function.

\begin{lemma}
\label{lem:prop2}
For real $X$, $Y$ and $Z$ with
$$
X\ge 1, \qquad 2\le Z\le Y,\qquad X^2YZ<p,
$$
we consider the intervals $I=[1,X]$, $J=[1,Y]$ and denote by $\cP$
the set of the primes $z\in(Z/2,Z]$.
Then for $s\in\F_p^*$ the number of solutions of the congruence
\begin{equation}
\begin{split}
\label{eq:estxyz}
x_2z_2(s+x_1y_1)&\equiv  x_1z_1(s+x_2y_2)
\pmod p,\\
 x_1,x_2\in I,&\  y_1,y_2\in J,\
z_1,z_2\in \cP
\end{split}
\end{equation}
is at most $XYZp^{o(1)}$.
\end{lemma}

\begin{proof} First we consider the case $s+x_1y_1\equiv s+x_2y_2\pmod p$. Then
$x_1y_1 = x_2y_2$ since $1\le x_1y_1, x_2y_2\le XY<p$.

If, moreover,
$s+x_1y_1\equiv s+x_2y_2\equiv 0\pmod p$ then
the number $x_1y_1=x_2y_2$ is uniquely defined and so there are $p^{o(1)}$ possibilities
for  each of $ x_1,x_2,y_1,y_2$. Hence, the number of such solutions
is at most $Z^2p^{o(1)}$.

Assume now that
\begin{equation}
\label{eq:case equiv}
s+x_1y_1\equiv s+x_2y_2\not\equiv 0\pmod p.
\end{equation}
If also $x_1 = x_2$ then $y_1 = y_2$, $z_1 = z_2$, and we get $XY\#\cP\le XYZ$
solutions. If $x_1 \ne  x_2$, we specify the common value $u = x_1y_1 = x_2y_2$.
Then $ x_1,x_2,y_1,y_2$ are divisors of $u$ and so there are $p^{o(1)}$ possibilities
for  each of them.  For fixed $ x_1,x_2,y_1,y_2$
the ratio $z_1/z_2$ is uniquely defined modulo $p$ and $z_1/z_2\not\equiv1\pmod p$
(since  we have $x_1 \not \equiv  x_2\pmod p$ but $s+x_1y_1\equiv s+x_2y_2\pmod p$).
Therefore $z_1,z_2$ are now uniquely defined too, as they are primes not
exceeding $Z<\sqrt p$. Thus, the number of solutions
to~\eqref{eq:estxyz} satisfying~\eqref{eq:case equiv} is at most
$XYZ+XYp^{o(1)}$.

We now consider the case $s+x_1y_1\not \equiv s+x_2y_2\pmod p$.
Let $N$ be the number of solutions of~\eqref{eq:estxyz} with $s+x_1y_1\not \equiv s+x_2y_2\pmod p$.
In particular, this condition implies that $x_1z_1\not=x_2z_2$ and $x_1y_1 \ne  x_2y_2$.
We can assume that $N>2XYZ$, as otherwise there is nothing to prove. There exist integers $n_0, m_0$ with $1\le |n_0|< XZ, 1\le m_0\le Y$ such that we have at least $N/(2XYZ)$ solutions with $x_1z_1-x_2z_2=n_0, y_1=m_0$. From~\eqref{eq:estxyz} we see that
$$
x_1x_2(y_1z_2-y_2z_1)\equiv sn_0 \pmod p.
$$
Thus, the number $x_1x_2(y_1z_2-y_2z_1)$ is nonzero and well
defined modulo $p$. Since its absolute value does not exceed
$X^2YZ<p$, it  may take at most two different integer values.
Thus, we can retrieve $x_1,x_2$ and $y_1z_2-y_2z_1$ with
$p^{o(1)}$ possibilities. Once these numbers are retrieved, we use
the equality
$$
n_0m_0+(y_1z_2-y_2z_1)x_2=z_1(x_1y_1-x_2y_2)
$$
and retrieve $z_1$ with $p^{o(1)}$ possibilities. Consequently,
from $n_0=x_1z_1-x_2z_2$ we retrieve $z_2$, and then we retrieve
$y_2$ from~\eqref{eq:estxyz}.

Thus, $N/(2XYZ)\le p^{o(1)}$ and the result follows.
\end{proof}

\section{Background on Algebraic Integers}

Let $\K$ be a finite extension of $\Q$ and let $\ZK$ be the ring of
integers in $\K$. We recall that the logarithmic height of an
algebraic number $\alpha$ is defined as the logarithmic height
$H(P)$ of its minimal polynomial $P$, that is, the maximum logarithm
of the largest (by absolute value) coefficient of $P$.

We  need a bound of Chang~\cite[Proposition~2.5]{Chang0} on the
divisor function in algebraic number fields.

\begin{lemma}
\label{lem:Div ANF} Let $\K$ be a finite extension of $\Q$ of degree
$d = [\K:\Q]$. For any algebraic integer $\gamma\in \Z_K$ of
logarithmic height at most $H\ge 2$, the number of   pairs
$(\gamma_1, \gamma_2)$ of  algebraic integers $\gamma_1,\gamma_2\in
\Z_K$ of logarithmic  height at most $H$ with
$\gamma=\gamma_1\gamma_2$ is at most $\exp\(O(H/\log H)\)$, where
the implied constant depends on $d$.
\end{lemma}

Now recall that the Mahler measure of a nonzero polynomial
$$P(Z) = a_dZ^d +\ldots+ a_1Z + a_0
= a_d\prod_{j=1}^d(Z - \xi_j) \in\C[Z]
$$
is defined as
$$M(P)= |a_d|\prod_{j=1}^d \max\{1,|\xi_j|\},$$
see~\cite[Chapter~3, Section~3]{Mig}

We recall the following estimates, that follows immediately from a
much more general~\cite[Theorem~4.4]{Mig}:

\begin{lemma}
\label{lem:HeighMahl} For any nonzero polynomial $P$ of degree $d$
the following inequality holds
$$2^{-d}e^{H(P)} \le M(P)\le (d+1)^{1/2} e^{H(P)}.$$
\end{lemma}

\begin{cor}
\label{cor:twoPol} For any nonzero polynomials $Q_1,Q_2\in\C[Z]$ we
have
$$
H(Q_1Q_2)= H(Q_1) + H(Q_2) +O(1),
$$
where the implied constant depends only on $\deg Q_1$ and $\deg
Q_2$.
\end{cor}

\begin{lemma}
\label{lem:PolCoef} For any positive integer $\nu$ there is a
constant $C$ such that the following holds. If $P_1, P_2\in\Z[Z]$,
$P=P_1P_2$,
$$P(Z)=\sum_{j=0}^\nu u_j Z^{\nu-j}$$
and for some $A>0$ and $h>0$ the coefficients of the polynomial $P$
satisfy the inequalities
$$u_0\neq0,\qquad |u_j|\le Ah^j,\qquad j=0,\ldots,\nu,$$
then the polynomial $P_1$ has the form
$$P_1(Z)=\sum_{j=0}^\mu v_j Z^{\mu-j}$$
with
$$v_0\neq0,\quad|v_j|\le CAh^j\quad(j=0,\ldots,\mu).$$
\end{lemma}

\begin{proof} We construct the polynomials
$$Q(Z)=P(hZ),\quad Q_1(Z)=P_1(hZ), \quad Q_2(Z)=P_2(hZ).$$
We have
$$e^{H(Q)}\le Ah^\nu.$$
Moreover,
$$e^{H(Q_2)}\ge h^{\nu-\mu}$$
since the leading coefficient of $Q_2$ is at least $h^{\nu-\mu}$.
Therefore, by Corollary~\ref{cor:twoPol} we get
$$e^{H(Q_1)}\ll Ah^\mu,$$
and the result follows.
\end{proof}

A particular case of Lemma~\ref{lem:PolCoef} is the following
statement (see, for example,~\cite[Theorem~6.32]{vzGG}).

\begin{lemma}
\label{lem:HeighDiv} Let $P,Q \in \Z[Z]$ be two univariate non-zero
polynomials with $Q\mid P$. If $P$ is of logarithmic height at most
$H\ge1$ then $Q$ is of  logarithmic height at most $H + O(1)$, where
the implied constant depends only on $\deg P$.
\end{lemma}

\section{Background on Geometry of Numbers}

Recall that a lattice in $\R^n$ is an additive subgroup of $\R^n$
generated by $n$ linearly independent vectors. Take an arbitrary
convex compact and symmetric with respect to $0$ body
$D\subseteq\R^n$. Recall that, for a lattice in $\Gamma\subseteq\R^n$
and $i=1,\ldots,n$, the $i$th successive minimum
$\lambda_i(D,\Gamma)$
of the set $D$ with respect to the lattice $\Gamma$ is defined as
the minimal number $\lambda$ such that the set $\lambda D$ contains
$i$ linearly independent vectors of the lattice $\Gamma$. Obviously,
$\lambda_1(D,\Gamma)\le\ldots\le\lambda_n(D,\Gamma)$. We need the
following result given in~\cite[Proposition~2.1]{BHW} (see
also~\cite[Exercise~3.5.6]{TaoVu} for a simplified form that is
still enough for our purposes).

\begin{lemma}
\label{lem:latp} We have,
$$\#(D\cap\Gamma)\le \prod_{i=1}^n \(\frac{2i}{\lambda_i(D,\Gamma)} + 1\).
$$
\end{lemma}

Using an obvious inequality
$$\frac{2i}{\lambda_i(D,\Gamma)} + 1 \le (2i+1)\max\left
\{\frac1{\lambda_i(D,\Gamma)},1\right\}$$
and denoting, as usual, by $(2n+1)!!$ the product of all odd positive numbers
up to $2n+1$,  we get the following

\begin{cor} \label{cor:latpoints} We have,
$$\prod_{i=1}^n \min\{\lambda_i(D,\Gamma),1\} \le (2n+1)!!
(\#(D\cap\Gamma))^{-1}.$$
\end{cor}

\section{Common Solutions to Many Quadratic Congruences with
Small Coefficients}

We need the following
statement, that can probably be extended in several directions.

\begin{lemma}
\label{lem:CommonSols2} For any positive integer $\nu\ge 3$ there
are numbers  $\eta>0$ and $C>0$, depending only on $\nu$, such that
if for a positive integer
$$
h\le \eta p^{1/\max\{\nu^2-2\nu-2,\nu^2-3\nu+4\}}
$$
and $s\in\F_p$  there are
$h^{\nu-1}$ different sequences $(A_1,\ldots,A_\nu)\in\Z^\nu$ with
$$
|A_i| < 2^i h^{i}, \qquad i =1, \ldots, \nu,
$$
such that
$$
A_1s^{\nu-1}+\ldots+A_{\nu-1}s+A_{\nu}\equiv 0 \pmod p,
$$
then we have the following:
\begin{itemize}
\item[(i)]
If $\nu=3$, then
$$
s\equiv a/b\pmod p
$$
for some integers $a,b$ with $a\ll h^{3/2}$, $b\ll h^{1/2}$.

\item[(ii)] If $\nu= 4$ then there is a nonzero sequence
$(B_2,B_3,B_4)\in\Z^3$ with
$$
 |B_{i}|<C h^{i-2},
\qquad i =2,3,4,
$$
and such that
$$
B_2s^2+B_3s+B_4\equiv 0 \pmod p.
$$

\item[(iii)] If $\nu\ge 5$ then there is a nonzero sequence
$(B_3,\ldots,B_\nu)\in\Z^{\nu-2}$ with
$$
 |B_{i}|<C h^{i-2-1/(\nu-2)},
\qquad i =3, \ldots, \nu,
$$
and such that
$$
B_3s^{\nu-3}+\ldots+B_{\nu-1}s+B_{\nu}\equiv 0 \pmod p.
$$

\end{itemize}
\end{lemma}

\begin{proof}
We can assume that $h\ge h_0(\nu)$ for some appropriate constant
$h_0(\nu)$, depending only on $\nu$. We define the lattice
$$\Gamma = \{(u_1,\ldots,u_\nu)\in\Z^\nu~:~
u_1s^{\nu-1}+\ldots+u_{\nu-1}s+u_\nu\equiv 0 \pmod p\}$$ and the
body
\begin{equation*}
\begin{split}
D = \{(u_1,\ldots,&u_\nu)\in\Z^\nu~:\\
&|u_1|<2^\nu h,\ldots, |u_{\nu-1}|<2^\nu h^{\nu-1},\ |u_{\nu}|<2^\nu
h^{\nu}\}.
\end{split}
\end{equation*}
We know that
$$\#(D\cap\Gamma)\ge h^{\nu-1}.$$
Therefore, by Corollary~\ref{cor:latpoints}, the successive minima
$\lambda_i=\lambda_i(D,\Gamma)$, $i=1,\ldots,\nu$, satisfy the
inequality
\begin{equation}
\label{ineqlambda} \prod_{i=1}^\nu\min\{1,\lambda_i\}\ll h^{1-\nu}.
\end{equation}
In particular, $\lambda_1\le1$.

We consider separately the following
seven cases.

{\it Case~1\/}: $\lambda_\nu\le1$. By definition of $\lambda_i$,
there are linearly independent vectors
$(u_1^i,\ldots,u_\nu^i)\in\lambda_iD\cap\Gamma$, $i=1,\ldots,\nu$.
By~\eqref{ineqlambda}, we have $\lambda_1\ldots\lambda_\nu\ll
h^{1-\nu}$. We consider the determinant
\begin{equation*}
\Delta = \det  \(
  \begin{array}{cccccccc}
    u_1^1 & \ldots & u_\nu^1\\
    \ldots & \ldots & \ldots\\
    u_1^\nu & \ldots & u_\nu^\nu\\
  \end{array}
\).
\end{equation*}
Clearly,
$$\Delta \ll h^{(\nu^2+\nu)/2}\lambda_1\ldots\lambda_\nu \ll h^{(\nu^2-\nu+2)/2}.
$$
On the other hand, from
$$u_1^is^{\nu-1} + \ldots + u_\nu^i \equiv 0\pmod p,\qquad i=1,\ldots,\nu$$
we conclude that $\Delta$ is divisible by $p$. Therefore, for a
sufficiently large $h_0(\nu)$ we derive $\Delta=0$, but this
contradicts linear independence of the vectors
$(u_1^i,\ldots,u_\mu^i)$, $i=1,\ldots,\nu$. Thus this case is
impossible.

{\it Case~2\/}:  $\lambda_{\nu-1}\le 1,\,\lambda_{\nu}>1$. We can
assume that $s\not\equiv 0\pmod p$. By definition, there are
linearly independent vectors
$(u_1^i,\ldots,u_{\nu}^i)\in\lambda_iD\cap\Gamma$,
$i=1,\ldots,\nu-1$. By~\eqref{ineqlambda}, we have
\begin{equation}
\label{ineqlambda nu-1} \lambda_1\ldots\lambda_{\nu-1}\ll h^{1-\nu}.
\end{equation}
Again by definition,
\begin{equation}
\label{eq:Case2SystemEqs^j} \left\{
\begin{array}{lll}
u_1^1s^{\nu-1}+\ldots+u_{\nu-1}^1s&\equiv &-u_{\nu}^1\pmod p;\\
&\ldots & \\
u_1^{\nu-1}s^{\nu-1}+\ldots+u_{\nu-1}^{\nu-1}s &\equiv
&-u_{\nu}^{\nu-1}\pmod p.
\end{array}\right.
\end{equation}
Let
\begin{equation*}
\Delta_0 = \det  \(
  \begin{array}{cccccccc}
    u_1^1 & \ldots & u_{\nu-1}^1\\
    \ldots & \ldots & \ldots\\
    u_1^{\nu-1} & \ldots & u_{\nu-1}^{\nu-1}\\
  \end{array}
\)
\end{equation*}
and for $i=1,2,\ldots,\nu-1$ let $\Delta_{\nu-i}$ be the determinant
of the matrix obtained from the matrix
\begin{equation*}
\(
  \begin{array}{cccccccc}
    u_1^1 & \ldots & u_{\nu-1}^1\\
    \ldots & \ldots & \ldots\\
    u_1^{\nu-1} & \ldots & u_{\nu-1}^{\nu-1}\\
  \end{array}
\)
\end{equation*}
by replacing its $i$-th column by $(-u_{\nu}^1,-u_{\nu}^2,\ldots,
-u_{\nu}^{\nu-1})$. From~\eqref{ineqlambda nu-1} we conclude that
\begin{equation}
\label{eq:Case2Xj} \Delta_j \ll h^{(\nu-2)(\nu-1)/2+j}, \qquad
j=0,1,\ldots, \nu-1.
\end{equation}
If $\Delta_0\equiv 0 \pmod p$, then $\Delta_0=0$ and from the system
of congruences~\eqref{eq:Case2SystemEqs^j} we derive that
$\Delta_j\equiv 0\pmod p$ and thus, in view of~\eqref{eq:Case2Xj},
we have $\Delta_j=0$ for all $1\le j\le \nu-1$. Hence, the rank of
the matrix
\begin{equation*}
\(
  \begin{array}{cccccccc}
    u_1^1 & \ldots & u_{\nu-1}^1 & u_{\nu}^1\\
    \ldots & \ldots & \ldots & \ldots\\
    u_1^{\nu-1} & \ldots & u_{\nu-1}^{\nu-1}& u_{\nu}^{\nu-1}\\
  \end{array}
\)
\end{equation*}
is strictly less than $\nu-1$, which contradicts to the linear
independence of the corresponding vectors.

Thus, we have that
$$
\Delta_0\not\equiv 0 \pmod p.
$$
Next, from the system~\eqref{eq:Case2SystemEqs^j} we find that
\begin{equation}
\label{eq:GeneralCase2s^j} s^j\equiv \Delta_j/\Delta_0\pmod p, \qquad
 j=1,\ldots, \nu-1.
\end{equation}
Since $s\not\equiv 0\pmod p$,  $\Delta_j\not\equiv 0\pmod p$.
Comparing $s$ and $s^2$ we obtain
$$
\Delta_1^2\equiv \Delta_0\Delta_2\pmod p.
$$
Since both hand sides are $O\(h^{\nu^2-3\nu+4}\)$, we see that
$$
\Delta_1^2=\Delta_0\Delta_2.
$$
Thus, there exist coprime integers integers $a,b$ such that for
$r_2=\gcd(\Delta_0,\Delta_2)$ we have
\begin{equation}
\label{eq:inductionj=2Xrab} \Delta_0=r_2b^2,\quad
\Delta_2=r_2a^2\quad \Delta_1=r_2ab.
\end{equation}

Now we claim that  there are integers $r_2,\ldots,r_{\nu-1}$ such
that the equalities
\begin{equation}
\label{eq:claim} \Delta_0=r_jb^j,\quad \Delta_1=r_jab^{j-1},\quad
\Delta_j=r_ja^j
\end{equation}
hold for all $j=2,3,\ldots, \nu-1$.

We prove this claim by induction on $j$. For $j=2$ the statement
follows from~\eqref{eq:inductionj=2Xrab}. We now assume
that~\eqref{eq:claim} holds for some $2\le j\le \nu-2$ and prove it
with $j$ replaced by $j+1$.

Comparing $s$ and $s^{j+1}$ in~\eqref{eq:GeneralCase2s^j}, we get
$$
\Delta_1^{j+1}\equiv \Delta_{j+1}\Delta_0^j\pmod p.
$$
We substitute here $\Delta_0$ and $\Delta_1$ in accordance to our
induction hypothesis. After cancellations (recall that
$\Delta_j\not\equiv 0\pmod p$), we get
\begin{equation}
\label{eq:cong raXb} r_ja^{j+1}\equiv \Delta_{j+1}b\pmod p.
\end{equation}
In view of the induction hypothesis, the left hand side
of~\eqref{eq:cong raXb} is of size at most
$$
|r_ja^j|^{(j+1)/j}\ll |\Delta_j|^{(j+1)/j}\ll
h^{((\nu-2)(\nu-1)/2+j)(j+1)/j}.
$$
Since
$$((\nu-2)(\nu-1)/2+j)(j+1)/j \le ((\nu-2)(\nu-1)/2+\nu-2)3/2
< \nu^2-3\nu+4,$$
we get
$$
|r_ja^j|^{(j+1)/j}\ll h^{\nu^2-3\nu+4}.
$$
Thus, we see  that the left hand side of~\eqref{eq:cong raXb} is
less than $p/2$. Again in view of the induction hypothesis we have
$|b|\le |\Delta_0|^{1/j}$. Hence, in view of~\eqref{eq:Case2Xj}, the
right hand side of~\eqref{eq:cong raXb} is
$$
\Delta_{j+1}b\ll
h^{(\nu-2)(\nu-1)/2+j+1}h^{(\nu-2)(\nu-1)/(2j)}\ll h^{\nu^2-3\nu+4}.
$$
Thus, again from the condition of the lemma we get that the right hand
side is less, than $p/2$. Hence, the congruence is converted to the
equality
$$
r_ja^{j+1}=\Delta_{j+1}b.
$$
Since $\gcd(a,b)=1$, this implies that for some integer $r_{j+1}$ we
have
$$
\Delta_{j+1}=r_{j+1}a^{j+1},\quad r_j=br_{j+1}.
$$
Replacing $r_j$ with its value given by the induction hypothesis, we
arrive to~\eqref{eq:claim}.

We have
$$
s\equiv \Delta_1/\Delta_0\equiv a/b\pmod p.
$$
In our intermediate statement we take  $j=\nu-1$ and
$$
|a|\le |\Delta_{\nu-1}|^{1/(\nu-1)}\ll h^{\nu/2},\quad |b|\le
|\Delta_0|^{1/(\nu-1)}\ll h^{\nu/2-1}.
$$
If $\nu=3$ then we are done. If $\nu=4$, then the statement
follows by taking $B_2=0, B_3=-a, B_4=b$.
If $\nu\ge 5$, then the statement
follows by taking
$$
B_{\nu}=-a, \quad B_{\nu-1}=b
$$
and $B_j=0$ for $j<\nu-1$.

Below we  use the following argument.
As in Cases~1 and 2, if $\lambda_r\le 1$ for some $r=1,\ldots,\nu$, then,
by definition there are linearly independent vectors
$(u_1^j,\ldots,u_{\nu}^j)\in\lambda_jD\cap\Gamma$,
$j=1,\ldots,r$. Clearly, we can assume that $\gcd (u_1^j,\ldots,u_{\nu}^j)=1$.
Next, we construct linear independent polynomials
$$P_j(Z)=\sum_{i=1}^\nu u_i^j Z^{\nu-i},\quad j=1,\ldots,r.$$
We note that $P_j(s)\equiv 0 \pmod p$ for $j=1,\ldots,r$.

{\it Case~3\/}: $\lambda_1\le 31 h^{-2}$ for $\nu\le4$ and
$\lambda_1\le h^{-2-1/(\nu-2)}$ for $\nu\ge5$.
For $\nu\le4$ we have $u_i^1\ll h^{i-2}$. Therefore, $u_1^1=0$
provided that $h_0(\nu)$ is large enough. If $\nu=3$ we take
$a=-u_3^1$, $b=u_2^1$. If $\nu\ge4$ we take
$B_i=u_i^1$ where $2\le i\le\nu$ for $\nu=4$ and $3\le i\le\nu$
for $\nu\ge5$.

We observe that if $\nu=3$ and $\lambda_2\ge1$ then, by
Corollary~\ref{cor:latpoints} we get $\lambda_1\le15 h^{-2}$.
Thus, for $\nu=3$ at least one of Cases~1--3 holds and the proof is complete.
Throughout the following we  always assume that $\nu\ge4$.

If Case~3 does not hold, then we have $\lambda_2\le1$ by~\eqref{ineqlambda}.
Thus, a polynomial $P_2$ is well-defined.

We denote $R_j=\gcd(P_1,P_j)$ if $\Res(P_1,P_j)=0$ for some $j>1$.
We have $R_j(s)\equiv 0 \pmod p$. If $R_j\neq\pm P_1$ then
$\deg R_j\le\deg P_1-1$ (taking into account that the coefficients
of $P_1$ are coprime). If, moreover, $\lambda_{\nu-1}>1$ then,
by~\eqref{ineqlambda}, we have
\begin{equation}
\label{eq:lam1est}
\lambda_1\ll h^{-1-1/(\nu-2)}.
\end{equation}
This inequality implies $u_1^1=0$, that is, $\deg P_1\le\nu-2$
(provided that $h_0(\nu)$ is large enough) and $u_i^1\ll h^{i-1-1/(\nu-2)}$ for
$i\ge2$. Therefore, if $R_j\neq\pm P_1$ and $\lambda_{\nu-1}>1$, then,
by Lemma~\ref{lem:PolCoef}, the coefficients of the polynomial $R_j$ satisfy
the statement of the theorem. Hence, we can suppose that $P_1$ divides $P_j$.

{\it Case~4\/}: $\lambda_1> 31 h^{-2}$, $\lambda_2\le 31 h^{-1}$
for $\nu= 4$ and $\lambda_1> h^{-2-1/(\nu-2)}$,
$\lambda_2\le h^{-1-1/(\nu-2)}$ for $\nu\ge5$; $\lambda_{\nu-1}>1$.

Suppose that $\nu=4$. We take $M=3$, $N=4$, $\sigma=\log(16\lambda_1)/\log h+2$,
$\vartheta=\log(16\lambda_2)/\log h+1$. Condition (i) of
Lemma~\ref{lem:neq:sigma} holds, and we can use Corollary~\ref{cor:DeterMagic}
taking into account that we have, by~\eqref{ineqlambda},
$\lambda_1\lambda_2\ll h^{-3}$. Hence,
$\Res(P_1, P_2)\ll h^8$ and  $|\Res(P_1, P_2)|<p$ provided that $\eta$
has been chosen small enough. On the other hand,  since
$P_1(s)\equiv P_2(s)\equiv 0\pmod p$,
$\Res(P_1, P_2)$ is divisible by $p$. Consequently, $\Res(P_1, P_2)=0$.
By our supposition $P_1$ divides $P_2$. Since $P_1$ and $P_2$
are linearly independent, we conclude that $\deg P_1\le \deg P_2-1$.
Using the inequality $\lambda_2\le 31 h^{-1}$ and Lemma~\ref{lem:PolCoef},
we see that the coefficients of the polynomial $R_j$ satisfy the statement of the theorem.

For $\nu\ge5$ the proof is similar. The inequality $\lambda_2\le h^{-1-1/(\nu-2)}$
implies $u_1^2=0$ provided that $h_0(\nu)$ is large enough. So,
$\deg P_1\le\nu-2$, $\deg P_2\le\nu-2$. Now we take $M=N=\nu-1$,
$\sigma=\log(2^\nu\lambda_1)/\log h+2$, $\vartheta=\log(2^\nu\lambda_2)/\log h+2$.
The condition~(iii) of Lemma~\ref{lem:neq:sigma} holds, and we can use
Corollary~\ref{cor:DeterMagic}. By~\eqref{ineqlambda}, we get
$\lambda_1\lambda_2\ll h^{-2-2/(\nu-2)}$. Therefore,
gives $\Res(P_1, P_2) \ll h^{\nu^2-2\nu-2}$. The rest is essentially
the same as for $\nu=4$.

Now suppose that $\nu=4$ and neither of Cases~1--4 holds. We conclude that
$$
\prod_{i=1}^4 \min\{\lambda_i(D,\Gamma),1\} > 31^2h^{-3}\ge
31^2(\#(D\cap\Gamma))^{-1}.
$$
However, this inequality contradicts Corollary~\ref{cor:latpoints}.
Thus, for $\nu=4$ at least one of Cases~1--4 holds and the proof is complete.
Throughout the following we  always assume that $\nu\ge5$.

In the rest of the proof we  estimate $\Res(P_1, P_j)$ for $j=2$
or $j=3$ by Corollary~\ref{cor:DeterMagic} considering that
$\lambda_1> h^{-2-1/(\nu-2)}$ and $\lambda_2> h^{-1-1/(\nu-2)}$.
We take $M=\nu-1$, $\sigma=\log(2^\nu\lambda_1)/\log h+2$.
If we know that $\deg P_j\le\nu-2$ then we take
$N=\nu-1$, $\vartheta=\log(2^\nu\lambda_1)/\log h+2$.
The inequality $\lambda_j> h^{-1-1/(\nu-2)}$ implies $\vartheta\ge0$,
and condition (ii) of Lemma~\ref{lem:neq:sigma} holds.
Otherwise, we take
$N=\nu$, $\vartheta=\log(2^\nu\lambda_1)/\log h+1$
and have the condition~(iii) of Lemma~\ref{lem:neq:sigma}.

{\it Case~5\/}: $\nu=5$, $\lambda_1> h^{-2-1/3}$, $\lambda_2>
h^{-1-1/3}$, $\lambda_4>1$. Assuming that $\lambda_3>1$ we get
contradiction with inequality~\eqref{ineqlambda} provided that
$h_0(5)$ is large enough. Hence, $\lambda_3\le1$. Using
again~\eqref{ineqlambda} we get
$$\lambda_1\ll h^{-4/3},\quad\lambda_1\lambda_j
\le\lambda_1\lambda_3\ll h^{-8/3}\quad (j=2,3).$$
Now we are in position to apply Corollary~\ref{cor:DeterMagic}
to the polynomials $P_1, P_j$ (recalling that $\deg P_1\le3$,
$\deg P_j\le4$). We get
$$\Res(P_1,P_j)\ll h^{23}(\lambda_1\lambda_j)^3\lambda_1\ll h^{41/3}.$$
Hence,  $|\Res(P_1, P_2)|<p$ provided that $h_0(5)$
has been chosen large enough. As before, we deduce that
$\Res(P_1,P_j)=0$. By our supposition, $P_1$ divides $P_2$ and $P_3$.
Since $P_1$, $P_2$ and $P_3$
are linearly independent, we conclude that $\deg P_1\le \deg P_j-2$
for $j=2$ or $j=3$. Using the inequality $\lambda_j\ll h^{-1/3}$ and
Lemma~\ref{lem:PolCoef}, we see that $P_1$ has the form $AZ^2+BZ+C$ where
$A \ll h^{2/3}$, $B\ll h^{5/3}$, $C\ll h^{8/3}$ as required.

Now the proof is complete for $\nu=5$.

{\it Case~6\/}: $\nu\ge6$, $\lambda_3\le h^{-1}2^{-\nu}$,
$\lambda_{\nu-1}>1$. Since $\lambda_j\le h^{-1}2^{-\nu}$
for $j=1,2,3$, we have $u_1^j=0$ and $\deg P_j\le\nu-2$.
We conclude from~\eqref{ineqlambda} that
$\lambda_1^2\lambda_3^{\nu-4}\ll h^{1-\nu}$. Hence,
$$\lambda_1\lambda_j\le\lambda_1\lambda_3\ll h^{-2-2/(\nu-2)}$$
for $j=2,3$. By Corollary~\ref{cor:DeterMagic} we have
$\Res(P_1, P_j)\ll h^{\nu^2-2\nu-2}$. As in the previous case, we
deduce that $\deg P_1\le \deg P_j-2$ for $j=2$ or $j=3$.
Using the inequality $\lambda_j\ll h^{-1}$ and
Lemma~\ref{lem:PolCoef}, we conclude that $\deg P_1\le\nu-4$
and $u_j^1\ll h^{j-3}$ for $j\ge4$, and the desired result follows.

{\it Case~7\/}: $\nu\ge6$, $\lambda_1> h^{-2-1/(\nu-2)}$,
$\lambda_2> h^{-1-1/(\nu-2)}$, $\lambda_3> h^{-1}2^{-\nu}$,
$\lambda_{\nu-1}>1$. Taking into account lower bounds for
$\lambda_2$ and $\lambda_j\,(j=4,\ldots,\nu-2)$ we
conclude from~\eqref{ineqlambda} that
$$\lambda_1\lambda_j\le\lambda_1\lambda_3\ll h^{-3+1/(\nu-2)}$$
for $j=2,3$. Next, using the lower bound for $\lambda_1$, we get
\begin{equation}
\label{eq:lam3est}
\lambda_3\ll h^{-(\nu-4)/(\nu-2)}.
\end{equation}
Assuming that $h_0(\nu)$ is large enough, we have $\lambda_3\le1$.
Hence, the polynomials $P_j$ are defined for $j\le3$;
moreover, $\deg P_1\le\nu-2$ and $\deg P_j\le\nu-1$ for $j=2,3$.
By Corollary~\ref{cor:DeterMagic}, we have for $j=2,3$
$$
\Res(P_1, P_j)\ll h^{\nu^2-2}(\lambda_1\lambda_3)^{\nu-1}\lambda_3^{-1}
\ll h^u,
$$
where
$$u=\nu^2-2-(3-1/(\nu-2))(\nu-1)+1=\nu^2-3\nu+3+1/(\nu-2)<\nu^2-3\nu+4.$$
As in the previous cases, we consequently conclude that $\Res(P_1,P_j)=0$,
$P_1$ divides $P_j$ for $j=2,3$ and $\deg P_1\le \deg P_j-2$ for $j=2$ or $j=3$.
Using~\eqref{eq:lam3est} and Lemma~\ref{lem:PolCoef} completes the proof.
\end{proof}

\begin{rem} One can try to separate the case
$$\lambda_{\nu-2}\le h^{-1}2^{-\nu},\quad \lambda_{\nu-1}>1$$
and to use the same arguments as in Case~2. We expect that
by this way it is possible to improve slightly the exponent in
the restriction on $h$ for $\nu\ge7$, however we have not
attempted to do so.
\end{rem}
 \section{Product Sets in $\F_p$}

Here we obtain some upper bounds on $K_{\nu}(p,h,s)$ that hold for
all primes.

\begin{theorem}
\label{thm:GProdSet} Let $\nu \ge 3$ be a fixed integer and let
$$
e_\nu=\max\{\nu^2-2\nu-2,\nu^2-3\nu+4\}.
$$
Then we have the bound
$$
K_{\nu}(p,h,s)\le  \(\frac{h^\nu}
{p^{\nu/e_\nu}} + 1\)
h^\nu\exp\(c(\nu)\frac{\log h}{\log\log h}\),
$$
where  $c(\nu)$ depends only on $\nu$.
\end{theorem}

\begin{proof}

Using Lemma~\ref{lem:Kss}, we see that we can assume that
\begin{equation}
\label{eq:small h} h< \eta p^{1/e_\nu}
\end{equation}
for some small constant $\eta>0$.

It is more convenient to include the case $\nu=2$. For $\nu=2$ we
know from~\cite{BGKS} the bound
\begin{equation}
\label{eq:nu=2} K_{\nu}(p,h,s)\le  h^2\exp\(c\frac{\log h}{\log\log
h}\),\qquad h\le p^{1/3},
\end{equation}
where  $c>0$ (see also Lemma~\ref{lem:GProdSet nu=2}). For
$\nu\ge3$ we  prove by induction on $\nu$ the estimate
\begin{equation}
\label{eq:nu>2} K_{\nu}(p,h,s)\le  h^\nu\exp\(c(\nu)\frac{\log
h}{\log\log h}\),\qquad h\le \eta_\nu p^{1/e_\nu}.
\end{equation}
By the induction hypothesis (the inequalities~\eqref{eq:nu=2} for
$\nu=3$ and~\eqref{eq:nu>2} for $\nu>3$), the set $(x_1,\ldots,
x_{\nu})$ of solutions of the congruence~\eqref{eq:cong x,y} for
which $x_i=y_j$ for some $i,j$ contributes to $K_{\nu}(p,h,s)$  at
most
\begin{equation}
\label{eq:Bad Sols2} h^{\nu}\nu^2 \exp\(c(\nu-1)\frac{\log
h}{\log\log h}\) \le h^{\nu}\exp\(0.5 c(\nu)\frac{\log h}{\log\log
h}\),
\end{equation}
provided that $h$ is large enough and we also choose $c(\nu) >
2c(\nu-1)$.

We associate with any solution of~\eqref{eq:cong x,y} such that
\begin{equation}
\label{eq:x neq y} \{x_1,\ldots, x_{\nu}\}\cap \{y_1,\ldots,
y_{\nu}\}=\emptyset,
\end{equation}
the polynomials
$$P(Z)=(x_1+Z)\ldots(x_{\nu}+Z),\qquad Q(Z)=(y_1+Z)\ldots(y_{\nu}+Z),
$$
and
\begin{equation}
\label{eq:Poly R} R(Z)= P(Z)-Q(Z).
\end{equation}
We note that each such  polynomial $R(Z)$ is nonzero and has a form
$$
R(Z)=A_1Z^{\nu-1}+\ldots+A_{\nu-1}Z+A_{\nu} \in \Z[Z],
$$
with $|A_i| \le 2^\nu h^{i}$, $i =1, \ldots, \nu$. In particular,
since $R(s)\equiv 0\pmod p$, it follows that $R(Z)$ is not a
constant polynomial.

Let $N$ be the number of the solutions of~\eqref{eq:cong x,y}
satisfying~\eqref{eq:x neq y}. We proceed as in the proof of
Lemma~\ref{lem:GProdSet nu=2}. By the pigeon-hole principle we have
at least $N/h$ solutions with the same $x_1=x_1^*$. We claim that
any polynomial $R$ induced by these solutions occurs at most
$\exp\(c_0(\nu) \log h/\log\log h\)$ times for some constant
$c_0(\nu)$ depending only on $\nu$. Indeed, fix $R$ and take
$Z=-x_1^*$. We get
\begin{equation}
\label{eq: repr M} M=-Q(-x_1^*)=z_1\ldots z_\nu,
\end{equation}
where $M=R(-x_1^*)$, $z_i=-x_1^*+y_i$, $i=1,\ldots,\nu$. The number
of solutions to~\eqref{eq: repr M} is bounded by $\exp\(c_0(\nu)\log
h/\log\log h\)$. Each solution determines the numbers
$y_1,\ldots,y_\nu$ and the polynomial $P$, and for each $P$ there
are at most $(\nu-1)!$ solutions of~\eqref{eq:cong x,y}. This proves
the claim.

Therefore, we can take
\begin{equation}
\label{eq: N_1 lower est} N_1\ge \exp\(-c_0(\nu)\frac{\log
h}{\log\log h}\)h^{-1}N
\end{equation}
solutions of~\eqref{eq:cong x,y} satisfying~\eqref{eq:x neq y} with
$x_1=x_1^*$ and distinct polynomials $R$. Assume that
$$
N_1\ge h^{\nu-1}.
$$
as otherwise there is nothing to prove (if we take $c(\nu) >
c_0(\nu) + 1$). Now we are in position to use
Lemma~\ref{lem:CommonSols2}.

If $\nu=3$, then we have
$$
s\equiv a/b\pmod p
$$
for some integers $a,b$ with $a\ll h^{3/2}$, $b\ll h^{1/2}$.

Note that for each solution $(x_1,x_2,x_3,y_1,y_2,y_3)$ that contributes to $N$, we have
\begin{equation*}
\begin{split}
0 \equiv (x_1+s)(x_2+s)(x_3+s)-(y_1&+s)(y_2+s)(y_3+s) \\
\equiv
(x_1+x_2+ x_3 & -y_1-y_2-y_3)s^2\\
 +(x_1x_2  +x_2&x_3+x_3x_1-y_1y_2-y_2y_3-y_3y_1)s\\
  +(x_1x_2&x_3-y_1y_2y_3) \pmod p.
\end{split}
\end{equation*}
Recalling that $s\equiv ab^{-1}\pmod p$, we now obtain
\begin{equation}
\label{eq:cong bc}
\begin{split}
(x_1+x_2+ x_3 -y_1-&y_2-y_3)a^2\\
 +(x_1x_2  +x_2&x_3+x_3x_1-y_1y_2-y_2y_3-y_3y_1)ab\\
  +(x_1&x_2x_3-y_1y_2y_3)b^2 \equiv 0\pmod p.
\end{split}
\end{equation}

Since the right hand side of~\eqref{eq:cong bc} is $\ll h^4$ we obtain
the equation
$$
(bx_1+a)(bx_2+a)(bx_3+a)=(by_1+a)(by_2+a)(by_3+a) + \lambda bp,
$$
where
$$
1\le x_i,y_i\le h, \qquad i =1, 2, 3,
$$
with some $\lambda \ll h^4/p + 1 \ll 1$. Recalling the well-known bound on the divisor function
(a special case of  Lemma~\ref{lem:Div ANF}) we obtain the result.

If $\nu\ge4$, then, by Lemma~\ref{lem:CommonSols2} we get a polynomial
$$R^*(Z)= B_2Z^{\nu-2}+\ldots+B_{\nu-1}Z+B_{\nu}$$
with $R^*(s)\equiv0 \pmod p$,
$$ |B_{i}|<2^\nu h^{i-2}, \qquad i =2,3,4,$$
for $\nu=4$ and
$$ B_2=0,\quad |B_{i}|<2^\nu h^{i-2-1/(\nu-2)},\qquad i =3, \ldots, \nu,$$
for $\nu\ge5$.

We fix such a polynomial $R^*$ and consider an arbitrary solution
of~\eqref{eq:cong x,y}, satisfying~\eqref{eq:x neq y} with
$x_1=x_1^*$ and take the corresponding polynomial $R$ given
by~\eqref{eq:Poly R}. Using Corollary~\ref{cor:DeterMagic}, and
recalling the assumption~\eqref{eq:small h}, we see that
$$
\Res(R,R^*)\ll p^8
$$
for $\nu=4$ and
$$\Res(R,R^*) \ll p^{\nu^2-2\nu-2-1/(\nu-2)}
$$
for $\nu\ge5$.
Thus,
\begin{equation}
\label{eq:Small Res} |\Res(R,R^*)|<p,
\end{equation}
provided that $\eta>0$ is small enough. Since
$$
R(s) \equiv R^*(s) \equiv 0 \pmod p
$$
we also have
\begin{equation}
\label{eq:Div Res} \Res(R,R^*) \equiv 0 \pmod p.
\end{equation}
Therefore, we see from~\eqref{eq:Small Res} and~\eqref{eq:Div Res}
that
$$
\Res(R,R^*) = 0.
$$
Hence, every polynomial $R$ has a common root with $R^*$. Thus,  by
Lemma~\ref{lem:HeighDiv}, we find an algebraic number $\beta$ of
logarithmic height $O(\log h)$  in an extension $\K$ of $\Q$ of
degree $[\K:\Q] \le \nu$ such that the equation
\begin{equation}
\label{eq:ConcentrationProdAlgebraic}
(x_1+\beta)\ldots
(x_{\nu}+\beta)= (y_1+\beta)\ldots (y_{\nu}+\beta)\neq0,
\end{equation}
where
$$1\le x_i,y_i\le h \mand x_1=x_1^*\not=y_i, \quad i=1, \ldots, \nu,
$$
has at least $N_1/\nu$ solutions. Now we have that
$$\beta =\frac{\alpha}{q},
$$
where $\alpha$ is an algebraic integer of height at most $O(\log h)$
and $q$ is a positive integer $q\ll h^{\nu}$, see~\cite{Nar}. From
the basic properties of algebraic numbers it now follows that the
numbers
$$qx_i+\alpha \mand qy_i+\alpha, \qquad i =1, \ldots, \nu,
$$
are algebraic integers of $\K$ of height at most $O(\log h)$.

Therefore,
 we conclude that for a sufficiently large $h$ the
 equation~\eqref{eq:ConcentrationProdAlgebraic}
 has at most
\begin{equation}
\label{eq:Good Sols} \exp\(C(\nu)\frac{\log h}{\log\log h}\) \le
\exp\(0.5 c(\nu)\frac{\log h}{\log\log h}\)
\end{equation}
 solutions, where $C(\nu)$ is the implied constant of  Lemma~\ref{lem:Div ANF}
and we also assume that $c(\nu) > 2C(\nu)$. This implies the
bound~\eqref{eq:nu>2}  and completes the proof.
\end{proof}

For a set $\cA \subseteq \F_p$ we denote
$$
\cA^{(\nu)} = \{a_1 \ldots a_\nu~:~a_1, \ldots, a_\nu \in \cA\}.
$$

\begin{cor}
\label{cor:ProdSet Fp}  Let $\nu \ge 3$ be a fixed integer and let
$$
e_\nu=\max\{\nu^2-2\nu-2,\nu^2-3\nu+4\}.
$$
Assume that for some
sufficiently large positive integer $h$ and prime $p$ we have
$$
h<p^{1/e_\nu}.
$$
For  $s \in \F_p$ we consider the set
$$
\cA = \left\{ x+s ~:~ 1\le x\le  h \right\} \subseteq \F_p.
$$
Then
$$
\# \cA^{(\nu)} > \exp\(-c(\nu) \frac{\log h  }{ \log \log h }\) h
^{\nu},
$$
where  $c(\nu)$ depends only on $\nu$.
\end{cor}

\section{Points on Exponential Curves}

This result improves the bound  of~\cite[Corollary~3]{CillGar}.

\begin{theorem}
\label{thm:ExpCong} Let $g$ be of multiplicative order $t$ modulo
$p$ and let $\gcd(a,p)=1$. Let  $I_1$ and $I_2$ be two intervals
consisting on $h_1$ and $h_2$ consecutive integers respectively, where $h_2\le t$.
Then the number
$R_{a,g,p}(I_1,I_2)$ of solutions of the congruence
$$
x\equiv ag^z \pmod p, \quad (x,z)\in I_1\times I_2
$$
is bounded by
$$
R_{a,g,p}(I_1,I_2) < \(h_1p^{-2/5} + 1\)h_2^{1/2+o(1)}.
$$
\end{theorem}

\begin{proof} By the pigeonhole principle, there exists an interval $I_{11}\subseteq I_1$ of length
$$
|I_{11}|=\min \{h_1,\, \fl{p^{2/5}}\}
$$
such that
$$
R_{a,g,p}(I_1,I_2) \le \left(\frac{2h_1}{p^{2/5}}+1\right)
R_{a,g,p}(I_{11},I_2),
$$
where $R_{a,g,p}(I_{11},I_2)$ is the number of solutions of the
congruence
\begin{equation}
\label{eq:R(I_11)} x\equiv ag^z \pmod p, \quad (x,z)\in
I_{11}\times I_2.
\end{equation}
It is enough to prove that for any fixed $\varepsilon > 0$ we have
$$
R_{a,g,p}(I_{11},I_2) < h_2^{1/2+\varepsilon}.
$$

Let $\cX\subseteq I_{11}$ be the set of $x$ for which the
congruence~\eqref{eq:R(I_11)}  is satisfied for some $z\in I_2$.
Let
$$
T(\lambda) = \#\{\lambda \in \F_p^*~:~ \lambda \equiv x_1x_2 \pmod p
\ \text{for some}\ x_1,x_2 \in \cX\}.
$$
Then obviously,
$$
\#\{\lambda ~:~ T(\lambda) > 0\} \le 2h_2.
$$
Hence, using the Cauchy inequality, we obtain
\begin{equation}
\label{eq:LB}
\begin{split}
\#\{x_1x_2 \equiv &y_1y_2\pmod p~:~ x_1,x_2, y_1,y_2 \in \cX\} \\
& =\sum_{\lambda \in \F_p^*}T(\lambda)^2 \ge (2h_2)^{-1}
\(\sum_{\lambda \in \F_p^*}T(\lambda)\)^2 =  \frac{(\#\cX)^4}{2h_2}.
\end{split}
\end{equation}

Assume that $\#\cX > h_2^{1/2+\delta}$ for some $\delta>0$
(otherwise there is nothing to prove). In this case
$$
\frac{|I_{11}|^3}{\# \cX}< \frac{\min\{h_2^3,
p^{6/5}\}}{h_2^{1/2+\delta}} =\min\{h_2^{5/2-\delta},
p^{6/5}h_2^{-1/2-\delta}\} = o(p).
$$
So Lemma~\ref{lem:GProdSet nu=2} applies and implies that
\begin{equation}
\label{eq:UB} \#\{x_1x_2 \equiv y_1y_2\pmod p~:~ x_1,x_2, y_1,y_2
\in \cX\} \ll (\#\cX)^2 h_2^{o(1)}.
\end{equation}
Clearly~\eqref{eq:LB} and~\eqref{eq:UB} contradict the assumption
$\#\cX > h_2^{1/2+\delta}$ and the result follows.
\end{proof}

In particular, Theorem~\ref{thm:ExpCong} extends the range $h \le
p^{1/3}$ given in~\cite{CillGar} up to $h \le p^{2/5}$ under which
the bound $R_{a,g,p}(I_1,I_2) = h^{1/2+o(1)}$ holds for
$h=h_1=h_2$.

\section{Double Character Sums Estimates}

We first point out the following improvement of the result from
Friedlander and Iwaniec~\cite[Theorem 3]{FrIw2}.

\begin{theorem}
\label{thm:9/20} Let $AB<p$, $B\le A$, $
\cA\subseteq [M,M+A]$ and $\cB\subseteq [N,N+B]$.
For any integer $r\ge 1$, for the sum
$$
S_\chi(\cA,\cB)=\sum_{a\in \cA}\sum_{b\in \cB}\chi(a+b)
$$
with a non-principal multiplicative character modulo a prime $p$,
we have
\begin{equation*}
\begin{split}
S_\chi(\cA,\cB)  \ll A^{1/2}(\#\cA)^{1/2}\#\cB\(\frac{A+Bp^{1/2r}}{A^2\#\cB}\)^{1/4r}&p^{1/8r+o(1)}\\
+(\#\cA)^{1/2}\(\#\cB\)^{1/2}&\(A+p^{1/2r}B\)^{1/2},
\end{split}
\end{equation*}
where the implied constant may depend only on $r$.
\end{theorem}

\begin{proof}
We follow the argument on~\cite[Page~371]{FrIw2} and denote by $\nu(u)$ the
number of solutions to the congruence
$$
a(b_1-b_2)^{-1} \equiv u \pmod p
$$
in integers $a$, $b_1$, $b_2$ with $|a - M-N|< 2A$ and $b_1, b_2 \in \cB$,
$b_1 \not \equiv b_2 \pmod p$.

Lemma~\ref{lem:prop1} yields
the following improvement of~\cite[Bounds~(10)]{FrIw2}:
\begin{equation}
\label{eqn:NewBound}
\sum_{u=0}^{p-1}\nu(u)^2\ll A \(\#\cB\)^3p^{o(1)}
\end{equation}
(instead of  $A B(\# \cB)^2p^{o(1)}$ on the right hand side, as in~\cite{FrIw2}).
Indeed,  the sum in~\eqref{eqn:NewBound} is equal to the number of solutions of the congruence
$$
a_1(b_1-b_2)\equiv a_2(b_3-b_4) \pmod p,
$$
where $a_1,a_2$ belong to an interval $I$ of length $4A$, $b_1,b_2,b_3,b_4 \in \cB$ and
$b_1\not=b_2, b_3\not =b_4$. Furthermore $\cB$ belongs to an interval of length $B$.
We fix $b_2, b_4 \in \cB$ and thus get that the number of solutions is less than
$(\#\cB)^2$ times the number of solutions of
\begin{equation}
\label{eqn:NewCong}
a_1\widetilde b_1\equiv a_2 \widetilde b_2 \pmod p,
\end{equation}
where $a_1,a_2 \in I$  and  $\widetilde b_1, \widetilde b_2 \ne 0$ belong to the union of two
``shifted'' sets
$\cB-{b_2}$ and $\cB-{b_4}$, respectively.
Applying  Lemma~\ref{lem:prop1}
we derive   the bound
$$
((\#\cB)^2 +A\#\cB)p^{o(1)}<A\#\cB p^{o(1)}
$$
(since $B<A$) on the number of solutions to~\eqref{eqn:NewCong},
which in turn yields~\eqref{eqn:NewBound}.

This yields instead of~\cite[Bound~(11)]{FrIw2}, the bound
\begin{equation*}
\begin{split}
S_\chi(\cA,\cB) \ll (\#\cA)^{1/2}  \(\# \cB\)^{1-1/4r}A^{1/4-1/4r} (A+ p^{1/2r}&B)^{1/4}p^{1/8r+o(1)}\\
& +(A\#\cA\#\cB)^{1/2}.
\end{split}
\end{equation*}
Concluding the argument as in~\cite{FrIw2} we obtain the result.
\end{proof}

Taking  $r=2$, we derive:

\begin{cor}
\label{cor:Improve9/20} Under the conditions of Theorem~\ref{thm:9/20} if
$\cA=\cB$, $A\le p^{1/2}$ and $\#\cA>p^{9/20+\varepsilon}$ for some $\varepsilon>0$, then we have
$$
S_\chi(\cA,\cA) \ll(\#\cA)^2p^{-\delta},
$$
where  $\delta>0$ depends only on $\varepsilon$.
\end{cor}

We recall that it has been noted in~\cite[Remark]{FrIw} that the bound~\eqref{eqn:NewBound}
and Corollary~\ref{cor:Improve9/20} hold under some additional
conditions. So here we recover the same estimates without that restriction.

\section{Character Sums with the Divisor Function}

Next, we consider the sum
\begin{equation}
\label{sumdivfunc}
S_a(N) = \sum_{1\le n\le N}\tau(n)\chi(a+n),
\end{equation}
where $a\in\Z$, $\tau$ is the divisor function
and $\chi$ is a non-principal multiplicative character modulo a prime $p$.

Karatsuba~\cite{Kar1} has established a non-trivial estimate for~\eqref{sumdivfunc}
uniformly over the integers $a$ with  $\gcd(a,p)=1$
provided that
$N\ge p^{1/2+\eps}$ with some fixed $\eps>0$.
Chang~\cite{Chang1} has extended this result to $N\ge p^{\rho+\eps}$ where
$$
\rho=\frac18(7-\sqrt{17})=0.359\ldots.
$$
Furthermore, Karatsuba~\cite{Kar4} has also shown that if $0<|a|\le p^{1/2}$
then the sums~\eqref{sumdivfunc} can be nontrivially estimated
already for  $N\ge p^{1/3+\eps}$. Here we show that  one has a
nontrivial estimate of $S_a(N)$ for $N\ge p^{1/3+\eps}$ and any integer
$a$ with $\gcd(a,p)=1$.

\begin{theorem}
\label{thm:1/3} For any $\eps>0$ there is $\delta>0$ such that
if $N\ge p^{1/3+\eps}$  then uniformly over the integers $a$ with
$\gcd(a,p)=1$,
$$
S_a(N)\ll Np^{-\delta}.
$$
\end{theorem}

\begin{proof}
We can assume that
$N<p^{1/2 + 0.1\eps}$ since otherwise the
result follows from the aforementioned result of
Karatsuba~\cite{Kar1}.
Let $X_0=\sqrt N$. We have
\begin{equation}
\begin{split}
\label{eq:estforS}
S_a(N)& =\sum_{1\le x\le X_0}\chi(a+x^2)+\sum_{1\le x\le X_0}
2\sum_{x<y\le N/x}\chi(a+xy)\\
&=2\sigma+O(\sqrt N),
\end{split}
\end{equation}
where
$$
\sigma=\sum_{1\le x< X_0}\sum_{x<y\le N/x}\chi(a+xy).
$$
Now we split the sum $W$ into $L=\fl{(\log X_0)/\log2}$ sums
$$
\sigma_j=\sum_{2^{j-1}\le x<\min(X_0,2^j)}\sum_{x<y\le N/x}\chi(a+xy),\quad j=1,\ldots,L.
$$

 We specify $j$ so
that for $\widetilde \sigma = \sigma_j$ we have
\begin{equation}
\label{eq:estforS1}
|\sigma| \ll  |\widetilde \sigma | L \ll  |\widetilde \sigma | \log p
\end{equation}
and define
$$I=\{x:\,2^{j-1}\le x<\min(X_0,2^j)\}.
$$
Furthermore, let
$$
\eta=\eps/4, \qquad
Z=Np^{-2\eta}2^{-j}, \qquad T=\fl{p^{\eta}}
$$
and let $\cP$ be the set of the primes $z\in(Z/2,Z]$.
Following~\cite{Kar4}, we observe that
\begin{equation}
\label{eq:estforS(j)}
\widetilde \sigma= \Sigma+O(Np^{-\eta}),
\end{equation}
where
$$ \Sigma=\frac1{\#\cP T}\sum_{x\in I}\sum_{x<y\le N/x}
\sum_{z\in \cP}\sum_{t=1}^T\chi(a+x(y+zt)).$$
We now prove that for a sufficiently large $p$ we have
\begin{equation}
\label{eq:estforS2desired}
| \Sigma|\le Np^{-2\delta}
\end{equation}
for some $\delta>0$ that depends only on $\eps$. Then the desired
result  follows from~\eqref{eq:estforS},  \eqref{eq:estforS1} and~\eqref{eq:estforS(j)}.

Defining
$$S(x,y,z)=\sum_{t=1}^T\chi(a+x(y+zt))$$
we have
\begin{equation}
\label{eq:estforS2}
|\Sigma|\le \frac1{\#\cP T}\sum_{x\in I}\sum_{x<y\le N/x}
\sum_{z\in \cP}|S(x,y,z)|.
\end{equation}

Assume that~\eqref{eq:estforS2desired} does not hold and define
$\cE$ as the set of triples $(x,y,z)$ involved in the summation in~\eqref{eq:estforS2}
and such
that $|S(x,y,z)|\ge Tp^{-3\delta}$. Then we have
$$
\# \cE \ge N\#\cP p^{-3\delta}
$$
provided that $p$ is large enough.
Using the multiplicativity of $\chi$ we derive
$$
|S(x,y,z)|=\left|\sum_{t=1}^T\chi(ax^{-1}z^{-1}+yz^{-1}+t)\right|,
$$
where $x^{-1}, z^{-1}$ are considered in $\F_p$, and define
$$\cU=\{ax^{-1}z^{-1}+yz^{-1}~:~(x,y,z)\in \cE\}.$$
Thus, we get
$$\left|\sum_{t=1}^T\chi(u+t)\right|\ge Tp^{-3\delta}$$
for any $u\in \cU$. Hence
\begin{equation}
\label{eq:estforuinU}
\sum_{u\in U}\left|\sum_{t=1}^T\chi(u+t)\right|\ge \#\cU Tp^{-3\delta}.
\end{equation}

Take $X=2^j$, $Y=N2^{1-j}$. Since we have assumed
that $N\le p^{1/2 + 0.1\eps}$, we have $X^2YZ<p$ provided that $p$ is large enough,
so we can use Lemma~\ref{lem:prop2}. Hence, we conclude that the congruence
$$
ax_1^{-1}z_1^{-1}+y_1z_1^{-1}\equiv  ax_2^{-1}z_2^{-1}+y_2z_2^{-1} \pmod p
$$
has at most $N\#\cP p^{o(1)}$ solutions in $(x_1,y_1,z_1), (x_2,y_2,z_2)\in \cE$.
Therefore, recalling that $\eta=\eps/4$ and  assuming that
$\delta < \eps/14$, we obtain
\begin{equation*}
\begin{split}
\#\cU& \ge (N\#\cP p^{-3\delta})^2 (N\#\cP p^{o(1)})^{-1}\gg NZ p^{-7\delta}
= N^{2}X^{-1} p^{-2\eta} p^{-7\delta} \\
&\gg  N^{3/2}  p^{-2\eta} p^{-7\delta}
\ge p^{1/2+3\eps/2 -2\eta -7\delta} =  p^{1/2+\eps -7\delta}
\ge p^{1/2+\eps/2} .
\end{split}
\end{equation*}
Therefore, by a result of Karatsuba
(see, for example,~\cite[p.~52]{Kar5})
we have
$$\sum_{u\in \cU}\left|\sum_{t=1}^T\chi(u+t)\right|\le \#\cU Tp^{-\kappa},
$$
where $\kappa > 0$ depends only on $\eps$.
 Taking $\delta < \min\{\kappa/3,\, \eps/14\}$
we see that~\eqref{eq:estforuinU} is false,
which concludes the proof.
\end{proof}

\begin{rem}  Let $\tau_k(n)$ be the number of ordered 
representations $n = d_1 \ldots d_k$ with positive integers $d_1, \ldots, d_k$.
Our argument can also be used to improve the range of $a$ of the result
of~\cite{Kar4} on analogues of the   sums $S_a(N)$ with $\tau_k$ instead
of $\tau$.
\end{rem}

\section*{Acknowledgement}

The research of  J.~B. was partially supported
by National Science Foundation   Grant DMS-0808042, that of  S.~V.~K. was partially supported
by Russian Fund for Basic Research, Grant N.~11-01-00329
and that of  I.~E.~S. by Australian Research Council Grant  DP1092835.

\end{document}